\documentclass[12pt]{article}
\usepackage[utf8]{inputenc}
\usepackage{graphicx}
\usepackage{amsmath}
\usepackage{blkarray}
\usepackage{fullpage}
\usepackage{mathtools}
\usepackage{commath}
\usepackage{csquotes}
\usepackage{amsfonts}
\usepackage{amssymb}
\usepackage{amsthm}
\newtheorem{thm}{Theorem}[section]
\newtheorem{lem}{Lemma}[section]

\newtheorem{dfn}{Definition}[section]
\newtheorem{cor}{Corollary}[section]

\newtheorem{prob}{Problem}[section]
\DeclareMathOperator{\diam}{diam}

\DeclareMathOperator{\rank}{rank}
\DeclareMathOperator{\Spec}{Spec}
\DeclareMathOperator{\MAX}{MAX}
\title{On the eccentricity matrices of trees: Inertia and spectral symmetry}
\author{Iswar Mahato \thanks{Department of Mathematics, Indian Institute of Technology Kharagpur, Kharagpur 721302, India. Email: iswarmahato02@gmail.com}\  \and M. Rajesh Kannan\thanks{Department of Mathematics, Indian Institute of Technology Kharagpur, Kharagpur 721302, India. Email: rajeshkannan@maths.iitkgp.ac.in, rajeshkannan1.m@gmail.com }}
\date{\today}
\begin{document}
\maketitle

\begin{abstract}
The \textit{eccentricity matrix} $\mathcal{E}(G)$ of a connected graph $G$ is obtained from the distance matrix of $G$ by keeping the largest non-zero entries in each row and each column, and leaving zeros in the remaining ones. The eigenvalues of $\mathcal{E}(G)$ are  the  \textit{$\mathcal{E}$-eigenvalues} of $G$. In this article, we find the inertia of the eccentricity matrices of trees. Interestingly,  any tree on more than $4$ vertices  with odd diameter has two positive and two negative $\mathcal{E}$-eigenvalues (irrespective of the structure of the tree). A tree with even diameter has the same number of  positive and negative $\mathcal{E}$-eigenvalues, which is equal to the number of 'diametrically distinguished' vertices (see Definition \ref{dia-distin}).   Besides we prove that the spectrum of the eccentricity matrix of a tree is symmetric with respect to the origin if and only if the tree has odd diameter. As an application, we characterize the trees with three distinct  $\mathcal{E}$-eigenvalues.
\end{abstract}



\section{Introduction}\label{sec1}
Throughout this paper, we consider simple, undirected  and connected graphs.  Let $G=(V(G),E(G))$ be a graph with  vertex set $V(G)=\{v_1,v_2,\hdots,v_n\}$ and  edge set $E(G)=\{e_1, \dots , e_m \}$. The \textit{adjacency matrix }of a graph $G$ on $n$ vertices, denoted by $A(G)$, is an $n \times n$ matrix whose rows and columns are indexed by the vertices of $G$ and the entries are defined as follows:  $A(G)=(a_{uv})$, where $a_{uv}=1$ if the vertices $u$ and $v$ are adjacent, and $a_{uv}=0$ otherwise. We fix an ordering for the vertex set, whenever we associate matrices with the given graph. The \emph{distance} between the vertices $u,v\in V(G)$, denoted by $d(u,v)$, is the minimum length of the paths between the vertices $u$  and $v$ in $G$, and define $d(u,u) =0$  for all $u \in V(G)$. The \emph{distance matrix} of $G$, denoted by $D(G)$, is an $n \times n$ matrix whose rows and columns are indexed by the vertices of $G$ and the entries are defined as follows:  $D(G)=(d_{uv})$, where $d_{uv}=d(u,v)$.  The \emph{diameter} of $G$, denoted by  $\diam(G)$, is the greatest distance between any pair of vertices  in $G$.  A \textit{diametrical path} is a path whose length equals to the diameter of $G$.

The \emph{eccentricity} $e(u)$ of a vertex $u$ is defined  as $e(u)=\max \{d(u,v):  v \in V(G)\}$.
The \textit{eccentricity matrix} $\mathcal{E}(G)$ of a connected graph $G$ on $n$ vertices is an $n\times n$ matrix with the rows and columns are indexed by the vertices of $G$, and the entries are defined as follows:
$${\mathcal{E}(G)}_{uv}=
\begin{cases}
\text{$d(u,v)$} & \quad\text{if $d(u,v)=\min\{e(u),e(v)\}$,}\\
\text{0} & \quad\text{otherwise.}
\end{cases}$$
In \cite{ran1,ran2},  Randi\'{c}  introduced the notion of eccentricity matrix of a graph, then known as $D_{\MAX}$-matrix. It was renamed as eccentricity matrix by Wang et al. in \cite{ecc-main}. The eigenvalues of $\mathcal{E}(G)$ are \textit{the $\mathcal{E}$-eigenvalues} of $G$.
Since $\mathcal{E}(G)$ is symmetric, all of its eigenvalues are real. Let $\xi_1>\xi_2>\hdots >\xi_k$ be all the distinct $\mathcal{E}$-eigenvalues of $G$. Then the $\mathcal{E}$-spectrum of $G$ is denoted by $\Spec_{\mathcal{E}}(G)$, and is defined as
\[ \Spec_{\mathcal{E}}(G)=
\left\{ {\begin{array}{cccc}
    \xi_1 & \xi_2  &\hdots & \xi_k\\
    m_1& m_2& \hdots &m_k\\
    \end{array} } \right\},
\]
where $m_i$ is the  multiplicity of $\xi_i$ for $i=1,2,\hdots,k$. As $\mathcal{E}(G)$ is entrywise nonnegative, by the Perron-Frobenius theorem, the spectral radius of $\mathcal{E}(G)$ is the largest eigenvalue of $\mathcal{E}(G)$. The spectral radius of $\mathcal{E}(G)$ is called the $\mathcal{E}$-spectral radius of $G$, and  is denoted by $\rho(\mathcal{E}(G))$.

Recently, the eccentricity matrices inspired much interest and attracted the attention of researchers. Wang et al. \cite{wang2020spectral} studied some spectral properties of graphs based on the eccentricity matrix. Mahato et al. \cite{mahato2020spectra} studied the spectra of eccentricity matrices of graphs. Recently, Wei et al. \cite{wei2020solutions} characterized the trees which have the minimum $\mathcal{E}$-spectral radius with the given diameter. As a consequence, they determined the trees on $n$ vertices with minimum $\mathcal{E}$-spectral radius, which solved a conjecture posted in \cite{ecc-main}. The eccentricity energy change of complete multipartite graphs due to an edge deletion is studied by Mahato and Kannan  \cite{mahato2022eccentricity}. Wei and Li \cite{wei2022eccentricity} established a relationship between the majorization and $\mathcal{E}$-spectral radii of complete multipartite graphs.  For more details about the eccentricity matrices of graphs,we refer to \cite{he2022largest,lei2022spectral,lei2021eigenvalues,qiu2022eccentricity, ran1,ran2, wang2020boiling,ecc-main,wei2022characterizing}.


From the spectral graph theory perspective, the eccentricity matrix is substantially different from the adjacency matrix and the distance matrix. The adjacency matrix and the distance matrix of a connected graph are always irreducible, while the eccentricity matrix of a connected graph need not be irreducible \cite{ecc-main}. For example, the eccentricity matrix of a complete bipartite graph of order $n$ and maximum degree less than $n-1$ is reducible. In contrast, the eccentricity matrix of a tree with at least two vertices is irreducible \cite{ecc-main}. The distance matrix of a tree is always invertible. However, the eccentricity matrix of a tree need not be invertible \cite{mahato2021spectral}.

\textit{The inertia} $In(M)$ of a symmetric matrix $M$ is the triple $\big(n_{+}(M),n_{-}(M),n_{0}(M)\big)$, where $n_{+}(M),$ $n_{-}(M)$ and $n_{0}(M)$ denote the number of positive, negative and zero eigenvalues of $M$, respectively. There are only a few graphs for which the inertias of eccentricity matrices are known. In \cite{mahato2020spectra}, Mahato et al. computed the inertia of eccentricity matrices of the path and the lollipop graphs. Recently, Patel et al. \cite{patel2021energy} studied the inertia of coalescence of graphs. One of the main objectives of this article is to compute the inertia of eccentricity matrices of trees. The inertia of any tree on more than $4$ vertices with an odd diameter is $(2,2,n-4)$ irrespective of the structure of the tree. The inertia of any tree with even diameter is $(l,l,n-2l)$, where $l$ is the number of 'diametrically distinguished' vertices (See Definition \ref{dia-distin}). This is done in Section \ref{sec3}.

For the adjacency matrix $A(G)$ of a graph $G$, it is known that the eigenvalues of $A(G)$ are symmetric about the origin if and only if $G$ is a bipartite graph. In Section \ref{sec4}, we consider this problem for the eccentricity matrices of trees and prove that the $\mathcal{E}$-eigenvalues of a tree $T$ is symmetric with respect to the origin if and only if $T$ is a tree with odd diameter.

\section{Preliminaries}
In this section, we introduce some of the needed notations and results. Let ${\mathbb {R}^{n\times n}}$ denote the set of all $n\times n$ matrices with real entries. For $A\in {\mathbb {R}^{n\times n}}$, let $A^{\prime}$, $\det (A)$, $\det(\lambda I-A)$ denote \textit{the transpose, determinant and characteristic polynomial of $A$}, respectively. Let $J$ and \textbf{0} denote the \textit{all one matrix} and \textit{all zero matrix} of appropriate sizes, respectively. For a real number $x$, $\lfloor x \rfloor$ denotes the greatest integer less than or equal to $x$, and $\lceil x \rceil$ denotes the least integer greater than or equal to $x$.

Let $A$ be an $n\times n$ matrix partitioned as
$ A=
\left( {\begin{array}{cc}
    A_{11} & A_{12} \\
    A_{21} & A_{22} \\
    \end{array} } \right)$,
where $A_{11}$ and $A_{22}$ are square matrices. If $A_{11}$ is nonsingular, then \textit{the Schur complement of $A_{11}$ in $A$}, denoted by $A/A_{11}$, is defined as $A_{22}-A_{21}{A_{11}^{-1}}A_{12}$. For Schur complement, we have $ \det A= (\det A_{11})\det(A_{22}-A_{21}{A_{11}^{-1}}A_{12})$.
Similarly, if $A_{22}$ is nonsingular, then \textit{the Schur complement of $A_{22}$ in A}, denoted by $A/A_{22}$, is $A_{11}-A_{12}{A_{22}^{-1}}A_{21}$, and we have $\det A= (\det A_{22})\det(A_{11}-A_{12}{A_{22}^{-1}}A_{21})$. In the following theorem, we state the well known Haynsworth inertia additivity formula.

\begin{thm}[{\cite[Theorem 1]{hayans}}]\label{Hyansworth}
Let $H$ be an $n\times n$ Hermitian matrix partitioned as
\[ H=
\left( {\begin{array}{cc}
    H_{11} & H_{12} \\
    H_{21} & H_{22} \\
    \end{array} } \right).\]
If $H_{11}$ is nonsingular, then $In(H)=In(H_{11})+In(H/H_{11})$, where $H/H_{11}$ is the Schur complement of $H_{11}$.
\end{thm}

A principal submatrix of a square matrix $A$ is the matrix obtained by deleting any $k$ rows and the corresponding $k$ columns from $A$. The determinant of the principal submatrix is called the principal minor of $A$.
\begin{thm}[{\cite[Page 53]{hor-john-mat}}]\label{ch-poly-minor}
Let $A$ be an $n \times n$ real matrix and let $E_k(A)$ denote the sum of its principal minors of size $k$. Then, the characteristic polynomial of $A$ is given by
$$\phi(\lambda)=\lambda^n-E_1(A){\lambda}^{n-1}+\hdots+(-1)^{n-1}E_{n-1}(A)\lambda+ (-1)^{n}E_{n}(A).$$
\end{thm}

We will need the following result about the interlacing of eigenvalues.
\begin{thm}[{\cite[Theorem 4.3.28]{hor-john-mat}}]\label{interlace}
 Let $M$ be a real symmetric matrix of order $m$, and let $N$ be its principal submatrix of order $n<m$. If $\lambda_1\geq \lambda_2\geq \hdots \geq \lambda_m$ are the eigenvalues of $M$ and $\mu_1\geq \mu_2\geq \hdots \geq \mu_n$ are the eigenvalues of $N$, then $\lambda_{m-n+i} \leq \mu_i \leq \lambda_i$ for $1\leq i \leq n$.
\end{thm}

Let $P_n$ and $K_{1,n-1}$ denote the path and the star graphs on $n$ vertices, respectively. The $\mathcal{E}$-spectrum of $K_{1,n-1}$ is given in the following lemma.
\begin{thm}[{\cite[Theorem 2.1]{mahato2020spectra}}]\label{star-spec}
Let $K_{1,n-1}$ be the star on $n$ vertices. Then, the $\mathcal{E}$-spectrum of $K_{1,n-1}$ is given by
\[\Spec_{\mathcal{E}}(K_{1,n-1})=
    \left\{ {\begin{array}{ccc}
        n-2+\sqrt{n^2-3n+3}  & n-2-\sqrt{n^2-3n+3} & -2  \\
        1 & 1 & n-2 \\
        \end{array} } \right\},\]
\end{thm}

An $n \times n$ nonnegative matrix $A $ is  \emph{reducible} if there exists an $n \times n$ permutation matrix $Q$ such that $QAQ^T =
\begin{pmatrix}
A_{11} & A_{12}  \\
0  & A_{22}
\end{pmatrix}$, where $A_{11}$ is a $r \times r$ sub matrix  with $1 \leq r < n$. If no such permutation matrix $Q$ exists, then $A$ is \emph{irreducible}.

\begin{thm}[{\cite[Theorem 2.1]{ecc-main}}]\label{ecc irreducible}
The eccentricity matrix $\mathcal{E}(T)$ of a tree with at least two vertices is irreducible.
\end{thm}

For $d\geq 3$ and $d$ is odd, let $T_{n,d}^{a,b}$ be the tree obtained from $P_{d+1}=v_0v_1v_2\hdots v_d$ by attaching $a$ pendant vertices  to $v_{\frac{d-1}{2}}$ and $b$ pendent vertices to $v_{\frac{d+1}{2}}$, where $a+b=n-d-1$ and $b\geq a\geq 0$. The following theorem gives the characterization of trees with minimum $\mathcal{E}$-spectral radius.

\begin{thm}[{\cite[Theorem 2.13]{wei2020solutions}}]\label{min-gen}
Let $T$ be a tree on $n$ vertices.
\begin{enumerate}
    \item If $4 \leq n\leq 15$, then
    $$\xi_1(T)\geq \sqrt{\frac{13n-35+\sqrt{(13n-35)^2-64(n-3)}}{2}} $$ with equality if and only if $T\cong T_{n,3}^{0,n-4}$.
    \item If $n\geq 16$, then $$\xi_1(T)\geq
    \begin{cases}
\text{$\sqrt{\frac{16n-21+\sqrt{800n-1419}}{2}}$,} & \quad\text{if $n$ is odd;}\\
\text{$\sqrt{\frac{16n-21+5\sqrt{32n-67}}{2}}$,} & \quad \text{if $n$ is even.}
\end{cases} $$
Each of the equalities holds if and only if $T\cong T_{n,5}^{\lfloor \frac{n-6}{2}\rfloor,\lceil \frac{n-6}{2}\rceil}$.
\end{enumerate}

\end{thm}

\section{Inertia of eccentricity matrices of trees}\label{sec3}
It is well known that the distance matrix $D(T)$ of a tree $T$ on $n$ vertices has exactly one positive eigenvalue and $n-1$ negative eigenvalues. Therefore,  $\rank(D(T))=n$ for any tree $T$ on $n$ vertices. In \cite{mahato2021spectral}, the authors proved that the star $K_{1,n-1}$ is the only tree of order $n$ for which the eccentricity matrix is invertible, that is,  $\rank(\mathcal{E}(K_{1,n-1}))=n$. Moreover, the inertia of $\mathcal{E}(K_{1,n-1})$ is $(1,n-1,0)$, as $D(K_{1,n-1})=\mathcal{E}(K_{1,n-1})$. In this section, we find the inertia of eccentricity matrices of all trees on $n$ vertices. First, let us start with computing the inertia of a block matrix.

\begin{lem}\label{inertia-B}
Let $B$ be a $2n \times 2n$ symmetric matrix partitioned as
\[B=
\left( {\begin{array}{cc}
    2d(J-I) & (2d-1)(J-I) \\
    (2d-1)(J-I) & \textbf{0} \\
    \end{array} } \right).\]
Then, the inertia of $B$ is $(n,n,0)$.
\end{lem}
\begin{proof}
Let $B_1=2d(J-I)$ and $B_2=(2d-1)(J-I)$. Then
$B= \left( {\begin{array}{cc}
    B_1 & B_2\\
    B_2 & \textbf{0} \\
\end{array} } \right).$
Since $B$ is a symmetric matrix and $B_1$ is invertible, by the Haynsworth inertia additivity formula, we have $In(B)=In(B_1)+In(0-B_2{B_1}^{-1}B_2)$. It is easy to see that the inertia of $B_1$ is $(1,n-1,0)$. We have $In(0-B_2{B_1}^{-1}B_2)=(n-1,1,0)$, because $B_2{B_1}^{-1}B_2=\frac{(2d-1)^2}{2d}(J-I)$. Thus $In(B)=(1,n-1,0)+(n-1,1,0)=(n,n,0)$.
\end{proof}

In the next theorem, we show that the rank and the inertia of the eccentricity matrices of trees with odd diameter do not depend on the structure of trees. For $U \subseteq V(T) $, let $R[U]$  denote the submatrix of $\mathcal{E}(T)$ with the rows are indexed by the elements of $U$.

\begin{thm}\label{Inertia-odd}
Let $T$ be a tree on $n\geq 4$ vertices with $\diam(T)=2d+1$, $d\in \mathbb{N}$. Then, the rank of $\mathcal{E}(T)$ is $4$. Moreover, $\mathcal{E}(T)$ has exactly two positive and two negative eigenvalues, that is, the inertia of $\mathcal{E}(T)$ is $(2,2,n-4)$.
\end{thm}

\begin{proof}
Let $T$ be a tree on $n\geq 4$ vertices with $\diam(T)=2d+1$, $d\in \mathbb{N}$. Since the diameter of $T$ is odd, $T$ has two centers, say, $u_0$ and $v_0$. Let $T_1$ and $T_2$ be the components of $T-\{u_0,v_0\}$ containing $u_0$ and $v_0$, respectively. Define
\begin{eqnarray*}
V_1 &=& \{ u\in T_1: d(u,u_0)=d\},\\
V_2 &=& \{ v\in T_2: d(v,v_0)=d\},\\
V_3 &=& \{ u\in T_1: 0\leq d(u,u_0)<d\}, \mbox{and}\\
V_4 &=& \{ v\in T_2: 0\leq d(v,v_0)<d\}.
\end{eqnarray*}

\noindent Then $V(T)=V_1 \cup V_2 \cup V_3 \cup V_4$, and $\{ V_1, V_2, V_3, V_4\}$ is a partition of $V(T)$.
It is easy to check the following
\begin{align*}
 \mathcal{E}(T)_{uv} &=2d+1,\quad \text{for all} \quad u\in V_1, v\in V_2, \\
 \mathcal{E}(T)_{uv} &=e(u), \quad \text{for all} \quad u\in V_4, v\in V_1,\\
 \mathcal{E}(T)_{uv} &=e(v), \quad \text{for all} \quad u\in V_2, v\in V_3,\\
 \mathcal{E}(T)_{uv} &=0,\quad \text{for all} \quad u,v\in V_i, i=1,2,3,4,\\
 \mathcal{E}(T)_{uv} &=0,\quad \text{for all} \quad u\in V_3, v\in V_1, \mbox{and}\\
 \mathcal{E}(T)_{uv} &=0,\quad \text{for all} \quad u\in V_2, v\in V_4.
\end{align*}

So corresponding to the above partition of $V(T)$, the eccentricity matrix $\mathcal{E}(T)$ of $T$ can be written as
\[ \mathcal{E}(T)=
\begin{blockarray}{ccccc}
& V_1 & V_2 & V_3 & V_4 \\
\begin{block}{c(cccc)}
 V_1 & \textbf{0} & (2d+1)J & \textbf{0} & P\\
 V_2 & (2d+1)J & \textbf{0} & Q & \textbf{0}\\
 V_3 & \textbf{0} & Q^{\prime} & \textbf{0} & \textbf{0}\\
 V_4 & P^{\prime} & \textbf{0} & \textbf{0} & \textbf{0}\\
\end{block}
\end{blockarray},\]
for some matrices $P$ and $Q$ of appropriate order. Note that, all the rows of $P$ are identical, and all the rows of $Q$ are identical.

For  $i=1,2$, the rows corresponding to every vertex in $V_i$ are identical in $\mathcal{E}(T)$, so  $\rank (R[V_1]) = \rank (R[V_2])=1$. Let $v\in V_3$. Then for every vertex $u \in V_3$,  there is  a scalar $c\in \mathbb{R}$ such that $R[v]=cR[u]$. Thus  $\rank(R[V_{3}])=1$. By the same argument, we get $\rank(R[V_4]) = 1$. Thus, the rank of $\mathcal{E}(T)$ is at most 4.

Let $v_i\in V_i$ for $i=1,2,3,4$. Then, the principal submatrix of $\mathcal{E}(T)$ indexed by $\{v_1,v_2,v_3,v_4\}$ is
\[A = \left( {\begin{array}{cccc}
0 & 2d+1 & 0 & 2d\\
2d+1 & 0 & 2d & 0 \\
0 & 2d & 0 & 0 \\
2d & 0 & 0 & 0\\
\end{array} } \right)\]
Since the rank of $A$ is $4$, $\rank(\mathcal{E}(T))\geq 4$. Thus $\rank(\mathcal{E}(T))=4$.

The eigenvalues of $A$ are $\frac{2d+1+\sqrt{(2d+1)^2+16d^2}}{2}$, $\frac{\sqrt{(2d+1)^2+16d^2}-(2d+1)}{2}$, $-\frac{\sqrt{(2d+1)^2+16d^2}-(2d+1)}{2}$, and  $-\frac{2d+1+\sqrt{(2d+1)^2+16d^2}}{2}$. That is, $A$ has two positive and two negative eigenvalues. Therefore, by Theorem \ref{interlace},  $\mathcal{E}(T)$ has at least two positive and two negative eigenvalues. But the rank of $\mathcal{E}(T)$ is $4$, hence $\mathcal{E}(T)$ has exactly two positive and two negative eigenvalues. Thus, the inertia of $\mathcal{E}(T)$ is $(2,2,n-4)$.
\end{proof}

\begin{dfn}\label{dia-distin}{\rm
    Let $T$ be a tree with  even diameter, and $u_0$ be the center of $T$. A vertex $v$ adjacent to $u_0$ in $T$ is \textit{diametrically distinguished} if there is a diametrical path contains the vertex $v$.}
    \end{dfn}
 In the following theorem, we show that the rank of $\mathcal{E}(T)$ is twice the number of diametrically distinguished vertices in $T$.

\begin{thm}\label{inertia-even}
Let $T$ be a tree on $n$ vertices with $\diam(T)=2d$, $d\geq 2, d\in \mathbb{N}$, and let $u_0$ be the center of $T$. Let $u_1,u_2,\hdots, u_l$ $(l\geq 2)$ be the neighbours of $u_0$ such that $e_{T_i}(u_i)=d-1$, where $T_i$ are the components of $T-\cup_{i=1}^l\{u_0,u_i\}$ containing $u_i$, for $i=1,2,\hdots, l$. Then, the inertia of $\mathcal{E}(T)$ is $(l,l,n-2l)$.
\end{thm}

\begin{proof}
Let $T_{l+1}$ be the component of $T-\cup_{i=1}^l\{u_0,u_i\}$ containing $u_0$. Let $V(T)=V_1\cup V_2\cup \hdots \cup V_{2l} \cup V_{2l+1}$ be a partition of the vertex set of $T$, where
\begin{eqnarray*}
V_i &=& \{ u\in T_i: d(u,u_0)=d\},\\
V_{l+i} &=& \{ v\in T_i:0 < d(u,u_0)<d\},~~~~~\mbox{ and}\\
V_{2l+1} &=& V(T_{l+1}).
\end{eqnarray*}

For each $i,j=1,2,\hdots, l$, it is easy to check that
\begin{align*}
 \mathcal{E}(T)_{uv} &=2d,\quad \text{for all} \quad u\in V_i, v\in V_j, i\neq j, \\
 \mathcal{E}(T)_{uv}&=e(u), \quad \text{for all} \quad u\in V_{l+i}, v\in V_j, i\neq j,\\
 \mathcal{E}(T)_{uv} &=e(v), \quad \text{for all} \quad u\in V_i, v\in V_{2l+1},\\
 \mathcal{E}(T)_{uv} &=0,\quad \text{for all} \quad u\in V_i, v\in V_{l+i},\\
 \mathcal{E}(T)_{uv} &=0,\quad \text{for all} \quad u\in V_{l+i}, v\in V_{2l+1},\\
 \mathcal{E}(T)_{uv} &=0,\quad \text{for all} \quad u,v\in V_i,\\
 \mathcal{E}(T)_{uv} &=0,\quad \text{for all} \quad u,v\in V_{l+i}, \mbox{and}\\
 \mathcal{E}(T)_{uv} &=0,\quad \text{for all} \quad u,v\in V_{2l+1}.
\end{align*}

Now, corresponding to the above  partition of $V(T)$, the eccentricity matrix of $T$ can be written as
\[ \mathcal{E}(T)=
\begin{blockarray}{cccccccccc}
& V_1 & V_2 &\hdots & V_l & V_{l+1} & V_{l+2} & \hdots & V_{2l} & V_{2l+1}\\
\begin{block}{c(ccccccccc)}
 V_1 & \textbf{0} & (2d)J & \hdots & (2d)J & \textbf{0} & P_{1,l+2} & \hdots & P_{1,2l} & P_{1,2l+1} \\
 V_2 & (2d)J & \textbf{0} & \hdots & (2d)J & P_{2,l+1} & \textbf{0} & \hdots & P_{2,2l} & P_{2,2l+1} \\
 \vdots & \vdots & \vdots & \ddots &\vdots &\vdots &\vdots &\ddots &\vdots & \vdots \\
 V_l & (2d)J & (2d)J  & \hdots & \textbf{0} & P_{l,l+1}  & P_{l,l+2} &\hdots & \textbf{0} & P_{l,2l+1} \\
 V_{l+1} & \textbf{0}  & P_{2,l+1}^{\prime} & \hdots & P_{l,l+1}^{\prime} & \textbf{0} & \textbf{0} & \hdots & \textbf{0} & \textbf{0}\\
 V_{l+2} & P_{1,l+2}^{\prime} & \textbf{0} & \hdots & P_{l,l+2}^{\prime} & \textbf{0} & \textbf{0} & \hdots & \textbf{0} & \textbf{0}\\
\vdots & \vdots & \vdots & \ddots &\vdots &\vdots &\vdots &\ddots &\vdots & \vdots \\
V_{2l} & P_{1,2l}^{\prime} & P_{2,2l}^{\prime} & \hdots & \textbf{0} & \textbf{0} & \textbf{0} & \hdots & \textbf{0} & \textbf{0}\\
V_{2l+1} & P_{1,2l+1}^{\prime} & P_{2,2l+1}^{\prime} & \hdots & P_{l,2l+1}^{\prime}& \textbf{0} & \textbf{0} & \hdots & \textbf{0} & \textbf{0}\\
\end{block}
\end{blockarray} ,
 \]
where the rows of each matrix $P_{i,l+j}$ are identical with $i\neq j$ and $i=1,2,\hdots, l$; $j=1,2,\hdots, l,l+1$. Note that, for a fixed $j$, the columns of the matrices $P_{i, l+j}$ are multiples of the vector $e=(1, \dots, 1)^T$ (of appropriate sizes).  Moreover, for $1 \leq j \leq l+1$ and  $v \in V_{l+j}$, all the non-zero entries of the column corresponds to the vertex $v$ are the same.

For $i=1,2,\hdots,l$, the rows corresponding to all the  vertices of $V_i$ are the same in $\mathcal{E}(T)$, so the rank of $R[V_i]$ is one. Let $u$ be a vertex in $ V_{l+i}$. Then for any vertex $v\in V_{l+i}$, $R[v]=cR[u]$, where $c\in \mathbb{R}$. Thus $\rank (R[V_{l+i}]) =1$ for $i=1,2,\hdots,l$.

If $v\in  V_{2l+1}$, then $R[v]=c_1R[v_1]+c_2R[v_2]+\hdots+c_lR[v_l]$ for some $v_i\in V_{l+i},  c_i\in \mathbb{R}$ and $i=1,2,\hdots,l$. Thus, the rows corresponding to $V_{2l+1}$ are linear combination of the rows corresponding to $V_{l+1},V_{l+2},\hdots,V_{2l}$ in $\mathcal{E}(T)$. Since $R[V_1],R[V_2],\hdots,R[V_{2l}]$ are of rank one in $\mathcal{E}(T)$, therefore the rank of $\mathcal{E}(T)$ is at most $2l$.

For $i=1,2,\hdots, l$, let $v_i\in V_i$ and $v_{l+i}$ be its adjacent vertex in $V_{l+i}$. Then the principal submatrix of $\mathcal{E}(T)$ indexed by the vertices $v_1,v_2,\hdots,v_l,v_{l+1},\hdots,v_{2l}$ is given by
\[B=
\begin{blockarray}{ccccccccc}
& v_1 & v_2 &\hdots & v_l & v_{l+1} & v_{l+2} & \hdots & v_{2l}\\
\begin{block}{c(cccccccc)}
 v_1 & 0 & 2d & \hdots & 2d & 0 & 2d-1 & \hdots & 2d-1  \\
 v_2 & 2d & 0 & \hdots & 2d & 2d-1 & 0 & \hdots & 2d-1  \\
 \vdots & \vdots & \vdots & \ddots &\vdots &\vdots &\vdots &\ddots &\vdots  \\
 v_l & 2d & 2d  & \hdots & 0 & 2d-1 & 2d-1 &\hdots & 0 \\
 v_{l+1} & 0  & 2d-1 & \hdots & 2d-1 & 0 & 0 & \hdots & 0 \\
 v_{l+2} & 2d-1 & 0 & \hdots & 2d-1 & 0 & 0 & \hdots & 0 \\
\vdots & \vdots & \vdots & \ddots &\vdots &\vdots &\vdots &\ddots &\vdots \\
v_{2l} & 2d-1 & 2d-1 & \hdots & 0 & 0 & 0 & \hdots & 0 \\
\end{block}
\end{blockarray}\]
Now, $B$ can be partitioned as $ B=
\left( {\begin{array}{cc}
    2d(J-I) & (2d-1)(J-I) \\
    (2d-1)(J-I) & \textbf{0} \\
\end{array} } \right)$. It is easy to see that the rank of $B$ is $2l$, and hence $\rank(\mathcal{E}(T))\geq 2l$. Thus the rank of $\mathcal{E}(T)$ is $2l$.

By Lemma \ref{inertia-B}, we have $In(B)=(l,l,0)$. Therefore, by Theorem \ref{interlace}, it follows that $\mathcal{E}(T)$ has at least $l$ positive eigenvalues and $l$ negative eigenvalues. Since $\rank(\mathcal{E}(T))=2l$, $\mathcal{E}(T)$ has exactly $l$ positive eigenvalues and $l$ negative eigenvalues. Thus the inertia of $\mathcal{E}(T)$ is $(l,l,2n-l)$.
\end{proof}

\section{Symmetry of $\mathcal{E}$-eigenvalues of trees }\label{sec4}

If $G$ is a simple graph, then the spectrum of the adjacency matrix of $G$ is symmetric with respect to the origin if and only if $G$ is bipartite.
In this section, we study the symmetry of the $\mathcal{E}$-spectrum of trees. We characterize the trees for which the $\mathcal{E}$-eigenvalues are symmetric with respect to the origin. Precisely, we show that if $T$ is a tree,  then the spectrum of $\mathcal{E}(T)$ is symmetric with respect to the origin if and only if $\diam(T)$ is odd.

 In the following theorem, to start with, we show that the $\mathcal{E}$-eigenvalues of a tree $T$ with odd diameter are symmetric with respect to the origin.

\begin{thm}\label{symmetry-odd}
If $T$ is a tree with odd diameter, then the eigenvalues of $\mathcal{E}(T)$ are symmetric about origin, that is, if $\lambda$ is an eigenvalue of $\mathcal{E}(T)$ with multiplicity $k$, then $-\lambda$ is also an eigenvalue of $\mathcal{E}(T)$ with multiplicity $k$.
\end{thm}
\begin{proof}
Let $T$ be a tree with diameter $2d+1$, $d\geq 1$. Let $u_0$ and $v_0$ be the centers of $T$, and $T_1$ and $T_2$ be the components of $T-\{u_0,v_0\}$ containing $u_0$ and $v_0$, respectively. Define
\begin{eqnarray*}
V_1 &=& \{ u\in T_1: d(u,u_0)=d\},\\
V_2 &=& \{ u\in T_1: 0\leq d(u,u_0)<d\},\\
V_3 &=& \{ v\in T_2: d(v,v_0)=d\}, ~\mbox{and}\\
V_4 &=& \{ v\in T_2: 0\leq d(v,v_0)<d\}.
\end{eqnarray*}
Then $V(T)=V_1 \cup V_2 \cup V_3 \cup V_4$ is a partition of $V(T)$. Partitioning the eccentricity matrix of $T$ with respect to the vertices of $V_1 \cup V_2$ and $V_3 \cup V_4$ gives us
$$\mathcal{E}(T) = \left( {\begin{array}{cc}
        \textbf{0} & A\\
        A^{\prime} & \textbf{0}\\
\end{array} } \right),$$
for some matrix $A$. The remainder of the proof is routine.
\end{proof}

As a consequence, we get the following result.
\begin{cor}\label{distinct-odd}
If $T$ is a tree $n\geq 5$ vertices with odd diameter, then $T$ has exactly five distinct $\mathcal{E}$-eigenvalues.
\end{cor}
\begin{proof}
If $T$ is a tree with odd diameter, then, by Theorem \ref{Inertia-odd}, it follows that $\mathcal{E}(T)$ has two  positive and two negative eigenvalues. If $\xi_1\geq \xi_2$ are two positive eigenvalues of $\mathcal{E}(T)$, then, by Theorem \ref{symmetry-odd}, $-\xi_1$ and $-\xi_2$ are also the eigenvalues of $\mathcal{E}(T)$. Since $\mathcal{E}(T)$ is a nonnegative irreducible matrix, by Perron-Frobenius theorem, $\xi_1$ is simple. Therefore, $\xi_1\neq \xi_2$ and hence $-\xi_1\neq -\xi_2$. Thus, all four non-zero eigenvalues of $\mathcal{E}(T)$ are distinct. This completes the proof.
\end{proof}

Let $\xi_n(T)$ denote the least $\mathcal{E}$-eigenvalue of a tree $T$. In \cite{mahato2020spectra}, Mahato et al. proved that $\xi_n(T)\leq -2$ if and only if $T$ is a star. In the following theorem, we  establish an upper bound for the least $\mathcal{E}$-eigenvalue of a tree $T$ with odd diameter.

\begin{thm}
Let $T$ be a tree on $n\geq 4$ vertices with odd diameter.
\begin{enumerate}
    \item If $4 \leq n\leq 15$, then
    $$\xi_n(T)\leq -\sqrt{\frac{13n-35+\sqrt{(13n-35)^2-64(n-3)}}{2}}. $$ Further, equality holds if and only if $T\cong T_{n,3}^{0,n-4}$.
    \item If $n\geq 16$, then $$\xi_n(T)\leq
    \begin{cases}
\text{$-\sqrt{\frac{16n-21+\sqrt{800n-1419}}{2}}$,} & \quad\text{if $n$ is odd;}\\
\text{$-\sqrt{\frac{16n-21+5\sqrt{32n-67}}{2}}$,} & \quad \text{if $n$ is even.}
\end{cases} $$
Each of the equality holds if and only if $T\cong T_{n,5}^{\lfloor \frac{n-6}{2}\rfloor,\lceil \frac{n-6}{2}\rceil}$.
\end{enumerate}
\end{thm}
\begin{proof}
The proof follows from Theorem \ref{min-gen} and Theorem \ref{symmetry-odd}.
\end{proof}

%
For an $n \times n $ matrix $A$,  let $E_k(A)$ denote the sum of its principal minors of size $k$. The following crucial result will be used in the proof of  Theorem \ref{symmetry-even}.
\begin{lem}\label{Pminor-A}
For $l\geq 2$, let $A$ be a $(2l+1)\times (2l+1)$  matrix defined by
\[A=
\begin{blockarray}{cccccccccc}
& v_1 & v_2 &\hdots & v_l & u_1 & u_2 & \hdots & u_l & u_0\\
\begin{block}{c(ccccccccc)}
 v_1 & 0 & 2d & \hdots & 2d & 0 & d+1 & \hdots & d+1 &d  \\
 v_2 & 2d & 0 & \hdots & 2d & d+1 & 0 & \hdots & d+1 & d \\
 \vdots & \vdots & \vdots & \ddots &\vdots &\vdots &\vdots &\ddots &\vdots & d  \\
 v_l & 2d & 2d  & \hdots & 0 & d+1 & d+1 &\hdots & 0 & d \\
 u_1 & 0 & d+1 & \hdots & d+1 & 0 & 0 & \hdots & 0 & 0 \\
 u_2 & d+1 & 0 & \hdots & d+1 & 0 & 0 & \hdots & 0 & 0 \\
\vdots & \vdots & \vdots & \ddots &\vdots &\vdots &\vdots &\ddots &\vdots & \vdots\\
u_l & d+1 & d+1 & \hdots & 0 & 0 & 0 & \hdots & 0 & 0\\
u_0 & d & d & \hdots & d & 0 & 0 & \hdots & 0 & 0\\
\end{block}
\end{blockarray}~.\]
Then the sum of all $(2l-1)\times (2l-1)$ principal minors of $A$ is non-zero, that is, $E_{2l-1}(A)\neq 0$.
\end{lem}
\begin{proof}
Let $M_{ij}$ denote the $(2l-1)\times (2l-1)$ principal submatrix of $A$ obtained by deleting the $v_i$-th and $v_j$-th rows and $v_i$-th and $v_j$-th columns. Now, we have the following two cases.

\textbf{Case 1.} If  $l=2$,  then
\[A = \left( {\begin{array}{ccccc}
0 & 2d & 0 & d+1 & d\\
2d & 0 & d+1 & 0 & d\\
0 & d+1 & 0 & 0 & 0\\
d+1 & 0 & 0 & 0 & 0\\
d & d  & 0 & 0 & 0\\
\end{array} } \right).\]
It is easy to check that $\det(M_{34})=4d^3$ is the only non-zero $3 \times 3$ principal minor of $A$. Hence, $E_3(A)=4d^3\neq 0$.

\textbf{Case 2.} Let $l\geq 3$. Note that the row corresponding to the vertex $u_0$ in the matrix $A$ is a linear combination of the rows corresponding to the vertices $\{u_1, \dots, u_l\}.$ Thus, if both $i$ and $j$ are in $v_k$ with $i\neq j$ and $i,j=1,2,\hdots,l$, then it is easy to see that  $\det(M_{ij})=0$. Let $i\in \{v_1, \dots, v_l\}$ and $j = u_0$. Then, in the matrix $M_{ij}$, the row corresponding to the vertex $u_i$ is a linear combination of the rows corresponding to the vertices $\{u_1, \dots, u_{i-1}, u_{i+1}, \dots,  u_l\}$, and hence $\det(M_{ij})=0$ for $i=1,2,\hdots,l$. Let $i\in \{v_1, \dots, v_l\}$ and $j \in \{u_1, \dots, u_l \}$. If $i =v_k$ and $j=u_k$ for some (same) $k$, then the row corresponding to the vertex $u_0$ is a linear combination of the rows corresponding to the vertices $\{u_1,\dots, u_{k-1}, u_{k+1}, \dots, u_l\}$. Suppose  $i = u_{k_1}$ and $j = v_{k_2}$ with $k_1 \neq k_2$, then the columns corresponding to the vertices $u_{k_2}$ and $u_0$ are linearly dependent, and hence $\det(M_{ij})=0$ for all $i,j=1,2,\hdots,l$.

\textbf{Claim:} If $i\in \{u_1, \dots, u_l\}$ and $j = u_0$, then $$\det(M_{ij}) = (-1)^{l-2}2d(l-1)(l-2)(d+1)^{2(l-1)}~\mbox{for}~ i=1,2,\hdots,l.$$

First let us consider the case  $i =u_1$ and $j = u_0$. Then \begin{align*}
M_{1j} &= \left( {\begin{array}{ccccccccc}
0 & 2d & 2d & \hdots & 2d & d+1 & d+1 & \hdots & d+1 \\
2d & 0 &  2d & \hdots & 2d & 0 & d+1  & \hdots & d+1 \\
\vdots & \vdots & \vdots & \ddots &\vdots &\vdots &\vdots &\ddots &\vdots\\
2d & 2d  & 2d & \hdots & 0 & d+1 & d+1 &\hdots & 0 \\
d+1 & 0 & d+1 & \hdots & d+1 & 0 & 0 & \hdots & 0 \\
d+1 & d+1 & 0 & \hdots & d+1 & 0 & 0 & \hdots & 0 \\
\vdots & \vdots & \vdots & \ddots &\vdots &\vdots &\vdots &\ddots &\vdots\\
d+1 & d+1 & d+1 & \hdots & 0 & 0 & 0 & \hdots & 0 \\
\end{array} } \right)   \\
&= \left( {\begin{array}{cc}
2d(J-I) & B^T\\
B & \textbf{0}\\
\end{array} } \right),
\end{align*}
where $B=\left( {\begin{array}{ccccc}
d+1 & 0 & d+1 & \hdots & d+1  \\
d+1 & d+1 & 0 & \hdots & d+1 \\
\vdots & \vdots & \vdots & \ddots &\vdots \\
d+1 & d+1 & d+1 & \hdots & 0  \\
\end{array} } \right).$

Since $2d(J-I)$ is a non-singular square matrix, by the Schur complement formula we have
\begin{align*}
\det(M_{1j}) &= \det(2d(J-I))\det(\textbf{0}-B(2d(J-I))^{-1}B^T) \\
&= (-1)^{l-2}2d(l-1)(l-2)(d+1)^{2(l-1)}
\end{align*}
It is easy to see that  $\det(M_{ij})$ are equal for every $i\in \{u_1, \dots, u_l\}$ and $j = u_0$, as they are all permutationally similar.

By a similar argument, if both $i $ and $j $ are $\{u_1,\dots, u_l\}$,  then $\det(M_{ij})=  (-1)^{l-2}4d^3(d+1)^{2(l-2)}$ for $i\neq j$ and $i,j=1,2,\hdots,l$.

Since, the signs of all $(2l-1)\times (2l-1)$ principal minors of $A$ are the same, therefore the sum of all $(2l-1)\times (2l-1)$ principal minors of $A$ is non-zero. This completes the proof.

%

\end{proof}

Next, we recall a well known condition for a matrix to have nonsymmetric  spectrum with respect to the origin.

\begin{lem}\label{symmetry-condition}
    Let $p(\lambda) = {\lambda}^n+c_1{\lambda}^{n-1}+\hdots+c_{n-1}\lambda+c_n$ be a  polynomial such that all of its roots are non-zero and real numbers. If  $c_i$ and $c_{i+1}$ are different from zero for some $i \in \{1, \dots, n-1\}$, then the roots of $p(\lambda)$ are not symmetric about the origin, that is, there exists a real number $\lambda_1$ such that $p(\lambda_1) =0$ and $p(-\lambda_1) \neq 0$.
\end{lem}
In the next theorem, we show that the coefficients of (at least) two consecutive terms in the characteristic polynomial of the eccentricity matrix of a tree with even diameter are non-zero. This implies that the spectrum is not symmetric with respect to the origin.

\begin{thm}\label{symmetry-even}
If $T$ is a tree with even diameter, then the eigenvalues of $\mathcal{E}(T)$ cannot be symmetric about the origin.
\end{thm}
\begin{proof}
The star $K_{1,n-1}$ is the only tree on $n$ vertices with diameter $2$, and the eigenvalues of $\mathcal{E}(K_{1,n-1})$ are not symmetric with respect to the origin.

Let $T$ be a tree on $n$ vertices with $\diam(T)=2d$ with $d\geq 2 $, and let $u_0$ be the center of $T$. Let $u_1,u_2,\hdots, u_l$ $(l\geq 2)$ be the neighbours of $u_0$ such that $e_{T_i}(u_i)=d-1$, where $T_i$ are the components of $T-\cup_{i=1}^l\{u_0,u_i\}$ containing $u_i$, for $i=1,2,\hdots, l$. Let $T_{l+1}$ be the component of $T-\cup_{i=1}^l\{u_0,u_i\}$ containing $u_0$. Now, corresponding to the partition of $V(T)$ mentioned in Theorem \ref{inertia-even}, the eccentricity matrix of $T$ can be written as
\[ \mathcal{E}(T)=
\begin{blockarray}{cccccccccc}
    & V_1 & V_2 &\hdots & V_l & V_{l+1} & V_{l+2} & \hdots & V_{2l} & V_{2l+1}\\
    \begin{block}{c(ccccccccc)}
        V_1 & \textbf{0} & (2d)J & \hdots & (2d)J & \textbf{0} & P_{1,l+2} & \hdots & P_{1,2l} & P_{1,2l+1} \\
        V_2 & (2d)J & \textbf{0} & \hdots & (2d)J & P_{2,l+1} & \textbf{0} & \hdots & P_{2,2l} & P_{2,2l+1} \\
        \vdots & \vdots & \vdots & \ddots &\vdots &\vdots &\vdots &\ddots &\vdots & \vdots \\
        V_l & (2d)J & (2d)J  & \hdots & \textbf{0} & P_{l,l+1}  & P_{l,l+2} &\hdots & \textbf{0} & P_{l,2l+1} \\
        V_{l+1} & \textbf{0}  & P_{2,l+1}^{\prime} & \hdots & P_{l,l+1}^{\prime} & \textbf{0} & \textbf{0} & \hdots & \textbf{0} & \textbf{0}\\
        V_{l+2} & P_{1,l+2}^{\prime} & \textbf{0} & \hdots & P_{l,l+2}^{\prime} & \textbf{0} & \textbf{0} & \hdots & \textbf{0} & \textbf{0}\\
        \vdots & \vdots & \vdots & \ddots &\vdots &\vdots &\vdots &\ddots &\vdots & \vdots \\
        V_{2l} & P_{1,2l}^{\prime} & P_{2,2l}^{\prime} & \hdots & \textbf{0} & \textbf{0} & \textbf{0} & \hdots & \textbf{0} & \textbf{0}\\
        V_{2l+1} & P_{1,2l+1}^{\prime} & P_{2,2l+1}^{\prime} & \hdots & P_{l,2l+1}^{\prime}& \textbf{0} & \textbf{0} & \hdots & \textbf{0} & \textbf{0}\\
    \end{block}
\end{blockarray} ,
\]
where the rows of each matrix $P_{i,l+j}$ are identical with $i\neq j$ and $i=1,2,\hdots, l$; $j=1,2,\hdots, l,l+1$.
By Theorem \ref{inertia-even}, $\rank(\mathcal{E}(T))=2l$ and hence the characteristic polynomial of $\mathcal{E}(T)$ can be written as

\begin{equation}
 \phi_{\mathcal{E}}(\lambda)={\lambda}^{n-2l}({\lambda}^{2l}+c_1{\lambda}^{2l-1}+\hdots+c_{n-2l+1}\lambda+c_{n-2l}).
\end{equation}
To prove that the eigenvalues of $\mathcal{E}(T)$ are not symmetric with respect to origin, by Lemma \ref{symmetry-condition} it suffices to show that $c_{n-2l}\neq 0$ and $c_{n-2l+1}\neq 0$. Note that  $c_{n-2l}\neq 0$, as $\rank(\mathcal{E}(T))=2l$. By Lemma \ref{ch-poly-minor}, we have $c_{n-2l+1}=E_{2l-1}$, where $E_{2l-1}$ is the sum of all $(2l-1)\times (2l-1)$ principal minors of $\mathcal{E}(T)$. Let $A[\alpha]$ denote the $(2l-1)\times (2l-1)$ principal submatrix of $\mathcal{E}(T)$ whose rows and columns are indexed by $\alpha :=\{\alpha_1,\alpha_2,\hdots,\alpha_{2l-1}\}$. Since $R[V_1],R[V_2],\hdots,R[V_{2l}],R[V_{2l+1}]$ are  all of rank one in $\mathcal{E}(T)$, therefore if $\alpha_i,\alpha_j \in V_k$ for $k=1,2,\hdots,2l+1$, then $\det(A[\alpha])=0$. Thus, to get a  non-zero principal minor of order $(2l-1)\times (2l-1)$, we have to choose at most one row and one column from each $V_k$ for $k=1,2,\hdots,2l+1$.
Therefore, if $\det(A[\alpha]) \neq 0$, then $A[\alpha]$ is a $(2l-1)\times (2l-1)$ principal submatrix of $A$, where

\[A=
\begin{blockarray}{cccccccccc}
& v_1 & v_2 &\hdots & v_l & u_1 & u_2 & \hdots & u_l & u_0\\
\begin{block}{c(ccccccccc)}
 v_1 & 0 & 2d & \hdots & 2d & 0 & d+1 & \hdots & d+1 &d  \\
 v_2 & 2d & 0 & \hdots & 2d & d+1 & 0 & \hdots & d+1 & d \\
 \vdots & \vdots & \vdots & \ddots &\vdots &\vdots &\vdots &\ddots &\vdots & d  \\
 v_l & 2d & 2d  & \hdots & 0 & d+1 & d+1 &\hdots & 0 & d \\
 u_1 & 0 & d+1 & \hdots & d+1 & 0 & 0 & \hdots & 0 & 0 \\
 u_2 & d+1 & 0 & \hdots & d+1 & 0 & 0 & \hdots & 0 & 0 \\
\vdots & \vdots & \vdots & \ddots &\vdots &\vdots &\vdots &\ddots &\vdots & \vdots\\
u_l & d+1 & d+1 & \hdots & 0 & 0 & 0 & \hdots & 0 & 0\\
u_0 & d & d & \hdots & d & 0 & 0 & \hdots & 0 & 0\\
\end{block}
\end{blockarray}~.\]
Now, by Lemma \ref{Pminor-A}, it follows that the sum of all  $(2l-1)\times (2l-1)$ principal minors of $A$ is non-zero. That is, the sum of all  $(2l-1)\times (2l-1)$ principal minor of $\mathcal{E}(T)$ is non-zero and hence $c_{n-2l+1}\neq 0$. This completes the proof.
\end{proof}

By combining Theorem \ref{symmetry-odd} and Theorem \ref{symmetry-even}, we obtain the following result.

\begin{thm}\label{symm-tree}
Let $T$ be a tree of order $n$. Then the eigenvalues of $\mathcal{E}(T)$ are symmetric about origin if and only if $T$ is a tree with odd diameter.
\end{thm}

In \cite{wang2020spectral}, the authors posed the following problem.
\begin{prob}[{\cite{wang2020spectral}}]
Characterize the graphs with small number of distinct $\mathcal{E}$-eigenvalues.
\end{prob}
Motivated by this, in the following theorem we characterize the trees with three distinct $\mathcal{E}$-eigenvalues.

\begin{thm}
Let $T$ be a tree on $n\geq 4$ vertices. Then, $T$ has exactly three distinct $\mathcal{E}$-eigenvalues if and only if it is a star.
\end{thm}
\begin{proof}
If $T\cong K_{1,n-1}$, then by Lemma \ref{star-spec}, we have
\[\Spec_{\mathcal{E}}(K_{1,n-1})=
    \left\{ {\begin{array}{ccc}
        n-2+\sqrt{n^2-3n+3}  & n-2-\sqrt{n^2-3n+3} & -2  \\
        1 & 1 & n-2 \\
        \end{array} } \right\}.\]
Therefore, the star $K_{1,n-1}$ has  three distinct $\mathcal{E}$-eigenvalues.

For $n=4$, $P_4$ and $K_{1,3}$ are the only trees on $4$ vertices. It is easy to see that $P_4$ has  four distinct $\mathcal{E}$-eigenvalues. So, let us assume that $n\geq 5$. From Corollary \ref{distinct-odd}, it follows that if $T$ is a tree on $n\geq 5$ vertices with odd diameter, then it has exactly five distinct $\mathcal{E}$-eigenvalues.

Let $T$ be a tree with even diameter other than star. From Theorem \ref{inertia-even}, it follows that $T$ has at least two positive and two negative $\mathcal{E}$-eigenvalues. Since $\mathcal{E}(T)$ is a nonnegative irreducible matrix, by the Perron-Frobenius theorem, $\xi_1$ is simple and hence $\mathcal{E}(T)$ has at least two distinct positive eigenvalues. Again, by Theorem \ref{inertia-even}, zero is always an eigenvalue of $\mathcal{E}(T)$. Since trace of $\mathcal{E}(T)$ is zero, it has at least one distinct negative eigenvalue. Thus, $T$ has at least four distinct $\mathcal{E}$-eigenvalues. This completes the proof.
\end{proof}

 We conclude this article with an observation and a problem to investigate in the future. In Theorem \ref{symm-tree}, we have proved that the $\mathcal{E}$-eigenvalues of a tree $T$ are symmetric about origin if and only if $T$ is a tree with odd diameter.  A graph  $G$ is a \textit{diametrical graph} if every vertex $v\in V(G)$ has a unique vertex $\bar{v}\in V(G)$ such that $d_G(v,\bar{v})=\diam(G)$. For a diametrical graph $G$ with diameter $d$, the eccentricity matrix of $G$ can be written as  $ \left( {\begin{array}{cc}
 0 & dI_k\\
 dI_k & 0\\
\end{array} } \right )$, and the  $\mathcal{E}$-spectrum  of $G$ is  $ \left \{ {\begin{array}{cc}
 d & -d\\
 k & k\\
\end{array} } \right \}$. Thus, the $\mathcal{E}$-eigenvalues of a diametrical graph $G$ with diameter $d$ are symmetric with respect to the origin. So, it is natural to propose the following problem for future study.

\begin{prob}
Characterize the graphs for which the $\mathcal{E}$-eigenvalues are symmetric with respect to the origin.
\end{prob}

\bibliographystyle{plain}
\bibliography{ecc-ref}
\end{document}